\documentclass[mathpazo]{csam}

\begin{document}

\title{Multitask kernel-learning parameter prediction method 
for solving time-dependent linear systems}

\author[K.~Jiang et~al.]{Kai Jiang\affil{1}\comma\corrauth,
       Juan Zhang\affil{1}~and Qi Zhou\affil{1}}
\address{\affilnum{1}\ Key Laboratory of Intelligent Computing and Information Processing of Ministry of Education, Hunan key Laboratory for Computation and Simulation in Science and Engineering, School of Mathematics and Computational Science,
 Xiangtan University, Xiangtan, Hunan, China, 411105, P.R. China.}
\emails{{\tt kaijiang@xtu.edu.cn} (K.~Jiang), {\tt zhangjuan@xtu.edu.cn} (J.~Zhang), {\tt qizhou@smail.xtu.edu.cn} (Q.~Zhou)}

\begin{abstract}
Matrix splitting iteration methods play a vital role in solving large sparse linear systems.
Their performance heavily depends on the splitting parameters, however, 
the approach of selecting optimal splitting parameters has not been well developed.
In this paper, we present a multitask kernel-learning parameter prediction method 
to automatically obtain relatively optimal splitting parameters, which contains simultaneous multiple parameters prediction and a data-driven kernel learning.
For solving time-dependent linear systems, including 
linear differential systems and linear matrix systems, 
we give a new matrix splitting Kronecker product method, as well as its convergence analysis and preconditioning strategy. 
Numerical results illustrate our methods can save an enormous amount of time in selecting the relatively optimal splitting parameters compared with the exists methods. 
Moreover, our iteration method as a preconditioner can effectively accelerate GMRES. 
As the dimension of systems increases, all the advantages of our approaches becomes significantly.
Especially, for solving the differential Sylvester matrix equation, the speedup ratio can reach tens to hundreds of times when the scale of the system is larger than one hundred thousand. 
\end{abstract}
\ams{62F15, 62J05, 65F08, 65F45, 65M22}
\keywords{Multitask kernel-learning parameter prediction, Time-dependent linear systems, Matrix splitting Kronecker product method, Convergence analysis, Preconditioning.}
	
\maketitle

\section{Introduction}

In this paper, we consider the time-dependent linear systems (TDLSs) of the form
\begin{equation}\label{eq:time-dependent_linear_system}
	\dot{\bm{x}}(t)=\mathcal{L}\circ \bm{x}(t),\quad t\in [0,T],
\end{equation}
where $\bm{x}(t):[0,T]\to V$, $V$ is $\mathbb{R}^n$ ($n\in \mathbb{N}$), $\bm{x}(0)=\bm{x}_0\in V$ is an initial value, and $\mathcal{L}$ is a linear operator. TDLSs appear in many branches of science and engineering, such as dynamical systems, quantum mechanics, semi-discretization of partial differential equations, differential matrix equations, etc \cite{gustafsson1995time, cirant2019existence, schneider2006parallel, mclachlan2001conformal, hauck2013collision,  amato2013necessary}. 
A series of discrete methods suitable for them have been developed, 
such as linear multistep schemes, Runge-Kutta methods, general linear methods, block implicit methods, and boundary value methods (BVMs) \cite{wanner1996solving, griffiths2010numerical, butcher2016numerical, gander2006optimized, lee2013laplace, brugnano1998solving, briley1980structure}. 
After temporal discretization, each TDLS can be transformed into a sparse linear system  
\begin{equation*}\label{eq:system}
	Q\bm{x}=\bm{b},\quad Q\in\mathbb{R}^{n\times n}~~\text{is nonsingular and}~~\bm{b}\in \mathbb{R}^n.
\end{equation*}

For solving linear systems, matrix splitting iteration methods play an important role as either solvers or preconditioners. 
The classic matrix splitting forms are all based on $Q=M-N$, 
where $M$ is a nonsingular matrix such that a linear system with the coefficient matrix $M$ can easily be solved, such as Jacobi, Gauss-Seidel, and successive over-relaxation iteration methods \cite{varga1962iterative}. 
Alternating direction implicit (ADI) schemes can effectively improve the performance by alternately updating approximate solution. 
They were initially designed to solve partial differential equations~\cite{peaceman1955numerical, douglas1956numerical}, and were gradually extended to more branches, including numerical algebra and optimization~\cite{varga1962iterative, lions1979splitting}.
The typical schemes in numerical algebra are Hermitian and skew-Hermitian splitting type methods~\cite{bai2003hermitian, 2007On, wang2013positive, benner2009adi}. 
Further, a general ADI (GADI) framework has recently been developed to put most existing ADI methods into a unified framework~\cite{jiang2022general}.

Matrix splitting iteration methods require coefficient matrix $Q$ into different parts with splitting parameters.
The convergence and performance of them are very sensitive to splitting parameters, therefore, choosing the optimal splitting parameters is critical. 
There have been several approaches to selecting splitting parameters. 
Experimental traversal method is limited due to an unbearable computational cost, especially for large-scale systems.
A more common approach is using theoretical analysis to estimate the bound of splitting parameters in a case-by-case way~\cite{bai2003hermitian, chen2014splitting}, while its performance heavily depends on the theoretical bound and the scale of systems. 
Recently, a data-driven approach, Gaussian process regression (GPR) method~\cite{jiang2022general}, has presented to predict optimal splitting parameters by choosing an appropriate kernel.
GPR method can efficiently predict one splitting parameter at one time.

However, matrix splitting iteration methods can have two or more splitting parameters, among which there are complicated links.
Independently predicting each splitting parameter would inevitably affect the predicted accuracy of the original GPR method.
Therefore, it requires improving the GRP method to predict multi-parameters simultaneously.
Another critical component that determines GPR's availability is kernel function~\cite{wilson2015kernel, wilson2016deep}.
The original GRP method chooses kernel function by the problem's properties~\cite{jiang2022general}, which might produce an artificial error.
Moreover, the chosen kernel in a kind of problems may be difficult to extend to others. 
Therefore, automatically learning kernel functions based on distinct problems is still a challenge.

In this work, we propose a new splitting parameter selection method for matrix spitting iteration methods. Concretely, we present multitask kernel learning (MTKL) method for obtaining the optimal splitting parameters which contains simultaneous multiple splitting parameters prediction and a data-driven kernel learning. Furthermore, we propose a new matrix splitting Kronecker product (MSKP) method to solve TDLSs \eqref{eq:time-dependent_linear_system}. Its convergence analysis and preconditioning strategy are also given. We apply our developed methods to some standard TDLSs, including two-dimensional  diffusion and convection-diffusion equations, and a differential Sylvester matrix equation. 
Numerical results illustrate our methods can save an enormous amount of time in selecting the relatively optimal splitting parameters compared with the exists methods. 
Moreover, our iteration method as a preconditioner can effectively accelerate GMRES. 
Especially for solving the differential Sylvester matrix equation, the speedup ratio can reach tens to hundreds of times when the scale of the system is larger than one hundred thousand. 
As the dimension of the systems increases, all the advantages of our approaches becomes significantly.

The rest of this paper is organized as follows.  
\Cref{sec:MTKL} proposes the MTKL method. 
\Cref{sec:appl} presents the MSKP method.
\Cref{sec:rslts} shows the efficiency and superiority of our proposed methods by numerical experiments. 
In \Cref{sec:conclusion}, we carry out the summary and give an outlook of future work.

\section{MTKL method}
\label{sec:MTKL}

In this section, we propose the MTKL method to predict the relatively optimal splitting parameters. MTKL approach contains a multi-task GPR method for multiple splitting parameters and a data-driven kernel learning method.

\subsection{Multitask GPR}

Matrix splitting iteration methods can have two or more splitting parameters which heavily affect the efficiency. Now we propose a practical multitask GPR method to simultaneously predict these relatively optimal splitting parameters.
The multitask method learns $M$ related functions $f_l$  $(l=1,\cdots,M)$ from training data $\{(\bm{x}_{li},y_{li})~|~l=1,\cdots,M,~i=1,\cdots,n,~\bm{x}_{li}\in \mathbb{R}^d,~y_{li}\in \mathbb{R}\}$. 
In practice, we consider the following noised model to avoid singularity:
\begin{equation}\label{eq:gpr_noise}
	y_{li}=f_l(\bm x_{li})+\epsilon_l,\quad \epsilon_l \sim \mathcal{N}(0,~\sigma^2_l),
\end{equation}
where $y_{li}$ $(\bm{x}_{li})$ is the $i$-th output (input) of the $l$-th task, $\epsilon_l$ is the white noise of the $l$-th task.

Let $\bm{y}=[y_{11},\cdots,y_{1n},\cdots,y_{M1},\cdots,y_{Mn}]^T=\text{vec}(Y^T)$ be the output vector, and $\bm{f}=[f_1,\cdots,f_1,\cdots,$ $f_M,\cdots,f_M]^T$ $=\text{vec}(F^T)$ be the latent function.
The multitask regression problem can be presented as a Gaussian process  prior over the latent function
\begin{equation*}
	\bm{f}\sim GP(\bm 0,~K^t \otimes K^x),
\end{equation*}
where $K^t\in \mathbb{R}^{M\times M}$ and $K^x\in \mathbb{R}^{n\times n}$ are the task and data covariance matrices, respectively \cite{bonilla2007multi}. The noised model \eqref{eq:gpr_noise} becomes
\begin{equation*}
	\bm{y}\sim \mathcal{N}(\bm{0},~K^t \otimes K^x + D\otimes I_n),
\end{equation*}
where $D=\text{diag}[\sigma_1^2,\cdots,\sigma_M^2]\in \mathbb{R}^{M\times M}$. The prediction distribution for the $l$-th task $\bm y_{*l}|\bm y$ on a new point $\bm x_*$ is
\begin{equation*}
	\bm y_{*l}|\bm y \sim \mathcal{N}(\bm{\mu}_{*l},~\bm{\Sigma}_{*l}),
\end{equation*}
where
	\begin{equation*}
		\begin{aligned}
	\bm{\mu}_{*l}&=(k^t_l\otimes k^x_{\bm{x}, \bm{x}_*})^T\bm{\Sigma}^{-1}\bm{y},\\
	\bm{\Sigma}_{*l}&=k^t_{ll}k^x_{\bm{x}_*, \bm{x}_*}-(k^t_l\otimes k^x_{\bm{x}, \bm{x}_*})^T\bm{\Sigma}^{-1}(k^t_l\otimes k^x_{\bm{x}, \bm{x}_*}),
\end{aligned}
\end{equation*}
$\bm{\Sigma}=K^t\otimes K^x+D\otimes I_n$. $k^t_{ll}$ and $k^t_l$ are the $l$-th diagonal element and the $l$-th column of $K^t$, respectively. $k^x_{\bm{x}, \bm{x}_*}$ is a covariances vector between test point $\bm{x}_*$ and training point $\bm{x}$, and $K^x$ denotes the covariance matrix of all training points. Hyperparameters $\theta^t$ and $\theta^x$ appear both in the task and data covariance functions. To obtain the hyperparameter estimation, we apply L-BFGS method to maximize marginal likelihood function $L$ in logarithmic form
\begin{equation*}
	\begin{aligned}
		&L=\log p(\bm{y}|\bm{x}, \theta_x, \theta_t)\\
		&=-\frac{n}{2}\log |K^t|-\frac{M}{2}\log |K^x|-\frac{1}{2} \mbox{trace}\left[(K^t)^{-1}F^T(K^x)^{-1}F\right]\\
		&-\frac{n}{2}\sum_{l=1}^M\log \sigma^2_l-\frac{1}{2}\mbox{trace}\left[(Y-F)D^{-1}(Y-F)^T\right]-\frac{Mn}{2}\log 2\pi.
	\end{aligned}
\end{equation*}

\subsection{Kernel learning approach}

Choosing an appropriate kernel function determines whether GRP method succeeds or not.
The original GRP method directly selects kernel function from the feature of problems~\cite{jiang2022general}. 
It might result in an artificial error and could be hardly extended to other problems. 
For example, Figure \ref{fig:diff_gpr_curve_bad} shows that the GPR method with a manual kernel function will produce a wrong regression curve. 
How to give an automatic way to choose kernel functions is worth investigating. 
Here, we present a data-driven method to learn the kernel function. Give a kernel library
$\mathcal{K}=\{k_{\xi}(x,x')\,|~\xi=1,\cdots,N \}$ that contains basic kernel functions and their multiplicative combinations.
For the $l$-th ($l=1,\cdots,M$) training task, the required kernel function is
the linear combination of library elements
\begin{align*}
	k(x,x')= \sum_{\xi=1}^{N} c_{l\xi} k_{\xi}(x,x').
\end{align*}
For $N$ training tasks, the weighted matrix is
\begin{equation*}
C=\begin{bmatrix}
c_{11} & c_{12} & \cdots & c_{1N} \\
c_{21} & c_{22} & \cdots & c_{2N} \\
\vdots & \vdots & & \vdots \\
c_{M1} & c_{M2} & \cdots & c_{MN} \\
\end{bmatrix}.
\end{equation*}
All weights can be obtained by training the data from concrete TDLS. A similar idea can be also found in pattern discovery~\cite{wilson2014covariance, wilson2013gaussian}.

\begin{remark}
	In GPR method~\cite{jiang2022general}, the training data comes from smaller systems which is a small data set.
	Learning kernel function by a (deep) neural network from small data
	is difficult. Therefore, we predetermine the kernel library
	instead of directly learning kernel function from data.
\end{remark}

\section{Solving TDLSs}
\label{sec:appl}

In this section, we propose a new matrix splitting iteration method, the MSKP method, for solving the TDLSs \eqref{eq:time-dependent_linear_system}. We also give the convergence analysis and preconditioning strategy.

\subsection{Discretization of TDLSs}
\label{subsec:discretization}

Here, we use BVMs to discretize TDLS \eqref{eq:time-dependent_linear_system} in  temporal direction to obtain a sparse linear system. 
We first give a brief description of BVMs. More details can refer to \cite{brugnano1998solving, burrage1995parallel}. For $i= k_1, \cdots, m-k_2$ ($k=k_1+k_2$), the $k$-step BVM is
\begin{equation}\label{eq:bvm}
	\sum_{j=-k_1}^{k_2}\alpha_{k_1+j}x_{i+j}=\tau \sum_{j=-k_1}^{k_2}\beta_{k_1+j}\mathcal{L}x_{i+j},
\end{equation}
where $\tau=T/m$ is step size, $x_i\approx x(t_i)$, $t_i=i \tau$, and $\alpha_j,~\beta_j,~j=-k_1,\cdots,k_2$ are parameters. The extra $k_1-1$ initial and $k_2$ final equations are
\begin{subequations}\label{eq:bvm_initial}
\begin{align*}
	&\sum_{j=0}^{k}\alpha_j^{i}x_j=\tau \sum_{j=0}^k\beta^{i}_j\mathcal{L}x_{j},\quad i=1,\cdots,k_1-1,\\
	&\sum_{j=0}^k\alpha_j^{i}x_{m-k+j}=\tau \sum_{j=0}^{k}\beta^{i}_j\mathcal{L}x_{m-k+j},\quad i=m-k_2+1,\cdots,m,
\end{align*}
\end{subequations}
where the coefficients $\alpha_j^{i}$ and $\beta_j^{i}$ are chosen such that the truncation errors over all node are consistent.
Applying BVM \eqref{eq:bvm} leads to the discretization form of 
TDLS \eqref{eq:time-dependent_linear_system} 
\begin{equation}\label{eq:bvm_linear_system}
	Q\bm{x}:=(A-\tau B\mathcal{L})\bm{x}=\bm{b},
\end{equation}
where $\bm{x}=(x_0,\cdots,x_m)^T$, $\bm{b}=\bm{e}_1x_0$, $\bm{e}_1=(1,0,\cdots,0)^T\in \mathbb{R}^{m+1}$, the matrices $A,~B\in \mathbb{R}^{(m+1)\times (m+1)}$ have the following structures
\begin{equation*}
	A=\begin{bmatrix}
		\begin{array}{ccccccc}
			1 & 0 & \cdots & 0\\
			\alpha^1_0 & \alpha^1_1 & \cdots & \alpha^1_k\\
			\vdots & \vdots &  & \vdots \\
			\alpha^{k_1-1}_0 & \alpha^{k_1-1}_1 & \cdots & \alpha^{k_1-1}_k\\
			& \alpha_0 & \alpha_1 & \cdots & \alpha_k\\
			& & \ddots & \ddots & & \ddots\\
			& & & \alpha_0 & \alpha_1 & \cdots & \alpha_k\\
			& & & \alpha^{m-k_2+1}_0 & \alpha^{m-k_2+1}_1& \cdots &\alpha^{m-k_2+1}_k\\
			& & & \vdots & \vdots &  & \vdots\\
			& & & \alpha^{m}_0 & \alpha^{m}_1 & \cdots & \alpha^{m}_k\\
		\end{array}
	\end{bmatrix},
\end{equation*}

\begin{equation*}
	B=\begin{bmatrix}
		\begin{array}{ccccccc}
			0 & 0 & \cdots & 0\\
			\beta^1_0 & \beta^1_1 & \cdots & \beta^1_k\\
			\vdots & \vdots &  & \vdots \\
			\beta^{k_1-1}_0 & \beta^{k_1-1}_1 & \cdots & \beta^{k_1-1}_k\\
			& \beta_0 & \beta_1 & \cdots & \beta_k\\
			& & \ddots & \ddots & & \ddots\\
			& & & \beta_0 & \beta_1 & \cdots & \beta_k\\
			& & & \beta^{m-k_2+1}_0 & \beta^{m-k_2+1}_1& \cdots &\beta^{m-k_2+1}_k\\
			& & & \vdots & \vdots &  & \vdots\\
			& & & \beta^{m}_0 & \beta^{m}_1 & \cdots & \beta^{m}_k\\
		\end{array}
	\end{bmatrix}.
\end{equation*}

In this work, we mainly consider linear differential systems and linear matrix systems. If $\mathcal{L}$ is a differential operator on space, TDLS becomes a linear differential system
\begin{equation}\label{eq:linear_differential_system}
	\begin{cases}
		\dfrac{\partial u}{\partial t} =\mathcal{L}u+f,&\text{in} \ \Omega \times [0,T],\\
		u=g(t),&\text{on} \ \partial\Omega  \times [0,T],\\
		u(\cdot,0)=\psi,&\text{in} \ \Omega,
	\end{cases}
\end{equation}
where $\Omega \subset \mathbb{R}^d$ with $d\geq 1$ is a bounded and open domain. $\psi$ is the initial condition and $f$ is the source term. 
By proper spatial discretization, such as finite difference and finite element methods, we can obtain an ordinary differential equation
\begin{equation}\label{eq:ode}
	\begin{cases}
		M\dot{U}(t)=-KU(t)+F(t), &t\in [0,T],\\
		U(0)=\Psi,
	\end{cases}
\end{equation}
where $U(t)\in \mathbb{R}^{n^d}$ contains approximate values of $u(\cdot,t)$ over spatial grid nodes. $n^d$ is the degree of freedom of spatial discretization.
$F(t)$ and $\Psi$ are similar notations corresponding to $f$ and $\psi$. $M,~K\in \mathbb{R}^{n^d\times n^d}$ are mass and stiff matrices, respectively.

For linear matrix systems, we consider the differential Sylvester matrix equation which has the following form
\begin{equation}\label{eq:linear_matrix_system}
	\begin{cases}
		\dot{X}(t)=\mathcal{A}X(t)+X(t)\mathcal{B}+\mathcal{E}\mathcal{F}^T,\quad t\in [0, T],\\
		X(t_0)=X_0,
	\end{cases}
\end{equation}
where $X(t)\in\mathbb{R}^{n\times n}$ for each $t\in [0,T]$, $\mathcal{A},~\mathcal{B}\in \mathbb{R}^{n\times n}$, and $\mathcal{E},~\mathcal{F}\in \mathbb{R}^{n\times s}$ are full rank matrices with $s\ll n$. 
The initial condition is $X_0=Z_0\widetilde{Z}^T$, $Z_0,~\widetilde{Z}_0\in \mathbb{R}^{n\times s}$.
It originate from many specific problems, such as dynamical systems, filter design theory, model reduction problems, differential equations, and robust control problems \cite{abou2012matrix, corless2003linear}. 
The equivalent ordinary differential equation of \eqref{eq:linear_matrix_system} is
\begin{equation}\label{eq:ode_sylvester}
	\begin{cases}
		M\dot{\bm{x}}(t)=-K\bm{x}(t)+\bm{f}(t),\quad t\in [0, T],\\
		\bm{x}(0)=\text{vec}(X_0),
	\end{cases}
\end{equation}
where $M$ is an $n^2\times n^2$ identity matrix, $K=-(I_n\otimes \mathcal{A}+\mathcal{B}^T\otimes I_n)$, $\bm{x}(t)=\text{vec}(X(t))$, and $\bm{f}(t)=\text{vec}\left(\mathcal{E}\mathcal{F}^T\right)$.

Applying BVM to \eqref{eq:ode}, we obtain a linear system
\begin{equation}\label{eq:linear_system}
	Q\bm{u}=(A\otimes M+\tau B\otimes K)\bm{u}=\bm{b},
\end{equation}
where $\bm{u}=[\Psi,U_1,\cdots,U_m]^T$, $\bm{b}=\tau (B\otimes I_{n^d})\bm{f}+\bm{e}_1\otimes \Psi$, $\bm{f}=[F_0,F_1,\cdots,F_m]^T$.
And using the BVM to \eqref{eq:ode_sylvester}, the corresponding linear system can also be obtained.

\subsection{MSKP method}

In this subsection, we propose the MSKP method and also give a fast version by the properties of Kronecker product.
Let $A = (a_{ij})\in \mathbb{R}^{m \times n},~B\in \mathbb{R}^{p \times q}$, the Kronecker product of $A$ and $B$ is defined by
\begin{equation*}
A \otimes B =\begin{bmatrix}
a_{11}B & a_{12}B & \cdots & a_{1n}B \\
a_{21}B & a_{22}B & \cdots & a_{2n}B \\
\vdots & \vdots & \vdots & \vdots \\
a_{m1}B & a_{m2}B & \cdots & a_{mn}B
\end{bmatrix}.
\end{equation*}
Lemma \ref{lem:kp} gives required properties of Kronecker product. More context about the Kronecker product can refer to \cite{horn94}.
\begin{lemma}[\cite{horn94}]\label{lem:kp}
Let $A,C\in \mathbb{R}^{m\times m}$, $B,D\in \mathbb{R}^{n\times n}$, then\\
(1)$\quad(A+C)\otimes (B+D)=A\otimes B+ A\otimes D+C\otimes B+C\otimes D$;\\
(2)$\quad(A\otimes B)(C\otimes D)=(AC)\otimes (BD)$;\\
(3)$\quad(A\otimes B)^{-1}=A^{-1}\otimes B^{-1}$;\\
(4)$\quad(A\otimes B) \bm x=\rm{vec}(BXA^T),~\rm{vec}(X)=\bm x$;\\
(5)$\quad \sigma(A\otimes B)=\{\lambda\mu |\lambda\in \sigma(A),\mu \in \sigma(B)\}$.\\
\end{lemma}
Here, the `$\text{vec}$' operator transforms matrices into vectors by stacking columns
\begin{equation*}
	X=[\bm x_1,\bm x_2,\cdots,\bm x_m]\in \mathbb{R}^{n\times m} \Longleftrightarrow \text{vec}(X)=[\bm x_1^T,\bm x_2^T,\cdots,\bm x_m^T]^T\in \mathbb{R}^{nm}.
\end{equation*}
Based on the properties of Kronecker product, we give the derivation process of MSKP scheme. 
First, we introduce generalized Kronecker product splitting (GKPS) method~\cite{chen2015generalized}
\begin{subequations}\label{eq:gkps}
	\begin{align}
		&((A+\alpha B)\otimes M)\bm u^{(k+\frac{1}{2})}=\bm b-(B\otimes(\tau K-\alpha M))\bm u^{(k)},\label{eq:gkpsa}\\
		&(B\otimes(\tau K+\beta M))\bm u^{(k+1)}=\bm b-((A-\beta B)\otimes M)\bm u^{(k+\frac{1}{2})}.\label{eq:gkpsb}
	\end{align}
\end{subequations}
Inspired by GADI framework \cite{jiang2022general}, we add a `viscosity' term to the right side of \eqref{eq:gkpsb}
\begin{equation*}
	\begin{aligned}
		&\bm b-(A-\beta B)\otimes M\bm u^{(k+\frac{1}{2})}+\frac{\alpha+\beta}{2} \omega B\otimes M(\bm u^{(k)}-\bm u^{(k+\frac{1}{2})})\\
		=&\bm b-(A-\beta B)\otimes M\bm u^{(k+\frac{1}{2})}+\frac{\alpha+\beta}{2}\omega B\otimes M\bm u^{(k)}-\frac{\alpha+\beta}{2}\omega B\otimes M\bm u^{(k+\frac{1}{2})}\\
		=&\bm b-(A+\alpha B)\otimes M\bm u^{(k+\frac{1}{2})}+\frac{\alpha+\beta}{2}\omega B\otimes M\bm u^{(k)}+\frac{\alpha+\beta}{2}(2-\omega) B\otimes M\bm u^{(k+\frac{1}{2})}\\
		=&B\otimes(\tau K-\alpha M)\bm u^{(k)}+\frac{\alpha+\beta}{2}\omega B\otimes M\bm u^{(k)}+\frac{\alpha+\beta}{2}(2-\omega)B\otimes M\bm u^{(k+\frac{1}{2})}\\
		=&B\otimes(\tau K+(\frac{\alpha+\beta}{2}\omega-\alpha)M)\bm u^{(k)}+\frac{\alpha+\beta}{2}(2-\omega)B\otimes M\bm u^{(k+\frac{1}{2})}.\\
	\end{aligned}
\end{equation*}
Then we obtain the MSKP scheme
\begin{equation}\label{eq:gadikp}
	\begin{cases}
		(A+\alpha B)\otimes M\bm u^{(k+\frac{1}{2})}=&\bm b-B\otimes(\tau K-\alpha M)\bm u^{(k)},\\
		B\otimes(\tau K+\beta M)\bm u^{(k+1)}=&B\otimes \Big(\tau K+\big(\dfrac{\alpha+\beta}{2}\omega-\alpha \big)M \Big)\bm u^{(k)}\\
		  & + \dfrac{\alpha+\beta}{2}(2-\omega) B\otimes M \bm u^{(k+\frac{1}{2})},
	\end{cases}
\end{equation}
where $k=0,1,\cdots$, the splitting parameters $\alpha,~\beta>0$ and $0\leq \omega<2$. Note that, when $\omega=0$, the MSKP method reduces to the GKPS scheme~\cite{chen2015generalized}, when $\omega=0$ and $\alpha=\beta$, the MSKP method becomes the Kronecker product splitting (KPS) approach~\cite{chen2014splitting}.

Eliminating the intermediate vector $\bm u^{(k+\frac{1}{2})}$, we can rewrite the MSKP method in a fixed point form
\begin{equation*}
	\bm u^{(k+1)}=T(\alpha,\beta,\omega)\bm u^{(k)}+P(\alpha,\beta,\omega)^{-1}\bm b,
\end{equation*}
where the iteration matrix is
\begin{equation}\label{eq:gadikp_iteration_matrix}
	\begin{aligned}
		&T(\alpha,\beta,\omega)
		=[(A+\alpha B)\otimes(\tau K+\beta M)]^{-1}\\
		&\cdot\bigg[(A-\beta B)\otimes (\tau K-\alpha M)+(A+\alpha B)\otimes \frac{\alpha+\beta}{2}\omega M+\frac{\alpha+\beta}{2}\omega B\otimes (\tau K-\alpha M)\bigg],
	\end{aligned}
\end{equation}
 and
\begin{equation}\label{eq:gadikp_preconditioner}
	P(\alpha,\beta,\omega)=\frac{2}{(\alpha+\beta)(2-\omega)}(A+\alpha B)\otimes(\tau K+\beta M).
\end{equation}
Obviously, there is a unique splitting
\begin{equation*}
	Q=P(\alpha,\beta,\omega)-R(\alpha,\beta,\omega),
\end{equation*}
where
\begin{equation*}
	\begin{aligned}
		R(\alpha,\beta,\omega)=&\frac{2}{(\alpha+\beta)(2-\omega)}\\
		&\cdot \bigg[(A-\beta B)\otimes (\tau K-\alpha M)+(A+\alpha B)\otimes \frac{\alpha+\beta}{2}\omega M+\frac{\alpha+\beta}{2}\omega B\otimes (\tau K-\alpha M)\bigg].
	\end{aligned}
\end{equation*}
Note that
\begin{equation}\label{eq:gadikp_TrelateP}
	T(\alpha,\beta,\omega)=P(\alpha,\beta,\omega)^{-1}R(\alpha,\beta,\omega)=I-P(\alpha,\beta,\omega)^{-1}Q.
\end{equation}
Therefore, the MSKP method \eqref{eq:gadikp} can become
\begin{equation*}
	\bm u^{(k+1)}=\bm u^{(k)}+P(\alpha,\beta,\omega)^{-1}\bm r^{(k)},\quad \bm r^{(k)}=\bm b-Q\bm u^{(k)}.
\end{equation*}

Using the properties of Kronecker product, we can obtain a fast implementation of MSKP method (see Algorithm \ref{alg:gadigkp}).
\begin{algorithm}
	\caption{MSKP method for solving linear system \eqref{eq:linear_system}.}
	\label{alg:gadigkp}
	\KwData{$\bm u^{(0)}\in \mathbb{R}^{n^dm}$, $\alpha, \beta>0$,  $0\leq \omega<2$}
	Calculate $\tau K+\beta M$ and $(A+\alpha B)^{-T}$\\
		\For{$k=0,1,2,\cdots$ until $\{\bm u^{(k)}\}$ converges}{
		$\bm r^{(k)}=\frac{(\alpha+\beta)(2-\omega)}{2}(\bm b-Q\bm u^{(k)})$\\
		$\text{vec}([\bm r^{(k)}_1,\bm r^{(k)}_2,\cdots,\bm r^{(k)}_m])=\bm r^{(k)}$\\	
		\For {$i=1,2,\cdots,m$}{
		$(\tau K+\beta M)\bm v^{(k)}_i \approx \bm r^{(k)}_i$ (use GMRES)
		}	
		$\bm v^{(k)}=\text{vec}\left([\bm v^{(k)}_1,\bm v^{(k)}_2,\cdots,\bm v^{(k)}_m]\cdot(A+\alpha B)^{-T}\right)$\\
		$\bm u^{(k+1)}=\bm u^{(k)}+\bm v^{(k)}$
		}
\end{algorithm}

\begin{remark}
Here we give a computational complexity analysis on MSKP method. For each iteration,
directly solving will result in $O\left((n^dm)^3\right)$, while Algorithm \ref{alg:gadigkp} only needs $O\left((n^dm)^2+n^dm^2+n^dm\right)$.
\end{remark}

\subsection{Convergence analysis of MSKP method}

In this subsection, we analyze the convergence of MSKP method. Let $\sigma(A)$ and $\rho(A)$ be the spectral set and the spectral radius of $A$, respectively. 

\begin{thm}\label{thm:convergence_of_gadikp}
	Assume that all eigenvalues of $B^{-1}A$ and  $M^{-1}K$
	have positive and nonnegative real parts, respectively. Then for any $\alpha>0$, $0<\beta\leq \tau \min_{\xi\in \sigma(M^{-1}K)}$ $Re(\xi)$ and $\omega\in [0,2)$, the MSKP method \eqref{eq:gadikp} converges to the unique solution $u^*$ of the linear system \eqref{eq:linear_system}. The spectral radius $\rho(T(\alpha,\beta,\omega))$ of iteration matrix $T(\alpha,\beta,\omega)$ satisfies
	\begin{equation*}
		\rho(T(\alpha,\beta,\omega))\leq\frac{1}{2}[(2-\omega)\varphi(\alpha,\beta)+\omega]<1,
	\end{equation*}
	where
	\begin{equation}\label{eq:gamma_alpha}
		\varphi(\alpha,\beta)=\max_{\lambda\in \sigma(B^{-1}A+\frac{\alpha-\beta}{2} I_m)}\left|\frac{\lambda-(\alpha+\beta)/2}{\lambda+(\alpha+\beta)/2}\right|.
	\end{equation}
\end{thm}

\begin{proof}
	Let $\tilde\lambda$ and $\tilde\mu$ be the eigenvalues of matrices $B^{-1}A$ and $\tau M^{-1}K$, respectively. By the property of Kronecker product, the eigenvalues of iteration matrix $T(\alpha,\beta,\omega)$ \eqref{eq:gadikp_iteration_matrix} have the following form
	\begin{equation*}
		p=\frac{(\tilde\lambda-\beta)(\tilde\mu-\alpha)+(\alpha+\beta)\omega(\tilde\lambda+\alpha)/2+\alpha+\beta\omega(\tilde\mu-\alpha)/2}{(\tilde\lambda+\alpha)(\tilde\mu+\beta)}.
	\end{equation*}
	Let $\alpha=\hat\alpha+\hat\beta$ and $\beta=\hat\alpha-\hat\beta$,
	the above equation becomes
	\begin{equation*}
	\scriptsize{
		p=\frac{(\tilde\lambda-\hat\alpha+\hat\beta)(\tilde\mu-\hat\alpha-\hat\beta)+\hat\alpha\omega(\tilde\lambda+\hat\alpha+\hat\beta)+\hat\alpha\omega(\tilde\mu-\hat\alpha-\hat\beta)}{(\tilde\lambda+\hat\alpha+\hat\beta)(\tilde\mu+\hat\alpha-\hat\beta)},}
	\end{equation*}
	which is equivalent to
	\begin{equation*}
		\begin{aligned}
			p=\frac{(\lambda-\hat\alpha)(\mu-\hat\alpha)+\hat\alpha\omega(\lambda+\hat\alpha)+\hat\alpha\omega(\mu-\hat\alpha)}{(\lambda+\hat\alpha)(\mu+\hat\alpha)}=\frac{(\lambda-\hat\alpha)(\mu-\hat\alpha)+\hat\alpha\omega(\lambda+\mu)}{(\lambda+\hat\alpha)(\mu+\hat\alpha)},
		\end{aligned}	
	\end{equation*}
	where $\lambda\in \sigma(B^{-1}A+\hat\beta I_m)$, $\mu\in \sigma(\tau M^{-1}K-\hat\beta I_m)$.

	Note that
	\begin{equation*}
		\begin{aligned}
			2p=&2\frac{\lambda-\hat\alpha}{\lambda+\hat\alpha}\cdot \frac{\mu-\hat\alpha}{\mu+\hat\alpha}+\omega\frac{2\hat\alpha(\lambda+\mu)}{(\lambda+\hat\alpha)(\mu+\hat\alpha)}\\
			=&(2-\omega)\frac{\lambda-\hat\alpha}{\lambda+\hat\alpha}\cdot \frac{\mu-\hat\alpha}{\mu+\hat\alpha}+\omega\frac{\lambda+\hat\alpha}{\lambda+\hat\alpha}\cdot \frac{\mu+\hat\alpha}{\mu+\hat\alpha}\\
			=&(2-\omega)\frac{\lambda-\hat\alpha}{\lambda+\hat\alpha}\cdot \frac{\mu-\hat\alpha}{\mu+\hat\alpha}+\omega.
		\end{aligned}
	\end{equation*}
	Denote $q=\dfrac{\lambda-\hat\alpha}{\lambda+\hat\alpha}\cdot \dfrac{\mu-\hat\alpha}{\mu+\hat\alpha}$, 
	\begin{equation*}
		p=\frac{1}{2}[(2-\omega)q+\omega].
	\end{equation*}
	Let $q=a+bi$ ($i=\sqrt{-1}$), 
	\begin{equation}\label{eq:pq}
		\begin{aligned}
			|p|=&\frac{1}{2}\left|(2-\omega)(a+bi)+\omega\right|\\
			=&\frac{1}{2}\sqrt{(2-\omega)^2a^2+\omega^2+2(2-\omega)\alpha\omega+(2-\omega)^2b^2}\\
			\leq &\frac{1}{2}\sqrt{(2-\omega)^2(a^2+b^2)+\omega^2+2(2-\omega)\omega\sqrt{a^2+b^2}}\\
			=&\frac{1}{2}\sqrt{(2-\omega)^2|q|^2+\omega^2+2(2-\omega)\omega|q|}\\
			=&\frac{1}{2}[(2-\omega)|q|+\omega].\\
		\end{aligned}
	\end{equation}
	If $0<\beta\leq \tau \min_{\xi\in \sigma(M^{-1}K)}Re(\xi)$ and $\alpha>0$, then
	\begin{equation*}
		\frac{\mu-\hat\alpha}{\mu+\hat\alpha}\leq 1,\quad \mu \in \sigma(\tau M^{-1}K-\hat\beta I_m).
	\end{equation*}
	When $0\leq\omega<2$, $Re(\lambda)>0$, 
	\begin{equation*}
		\begin{aligned}
			\rho(T(\alpha,\beta,\omega))
			=&\max_{\lambda\in \sigma(B^{-1}A+\hat\beta I_m)\atop \mu \in \sigma(\tau M^{-1}K-\hat\beta I_m)}|p|
			\leq \max_{\lambda\in \sigma(B^{-1}A+\hat\beta I_m)\atop \mu \in \sigma(\tau M^{-1}K-\hat\beta I_m)}\frac{1}{2}[(2-\omega)|q|+\omega]\\
			\leq & \frac{1}{2}\left[(2-\omega)\max_{\lambda\in \sigma(B^{-1}A+\hat\beta I_m)}\left|\frac{\lambda-\hat\alpha}{\lambda+\hat\alpha}\right|+\omega \right]\\
			=&\frac{1}{2}\left[(2-\omega)\max_{\lambda\in \sigma(B^{-1}A+\frac{\alpha-\beta}{2} I_m)}\left|\frac{\lambda-(\alpha+\beta)/2}{\lambda+(\alpha+\beta)/2}\right|+\omega \right]\\
			=&\frac{1}{2}[(2-\omega)\varphi(\alpha,\beta)+\omega]<1.
		\end{aligned}
	\end{equation*}	
	It implies that the MSKP method \eqref{eq:gadikp} converges to the unique solution $u^*$ of linear system \eqref{eq:linear_system}.
\end{proof}

\subsection{Accelerating GMRES by MSKP preconditioner}

Based on MSKP method, we give a preconditioning strategy to GMRES. From \eqref{eq:gadikp_TrelateP}, the linear system \eqref{eq:linear_system} is equivalent to
\begin{equation}\label{eq:linear_system_gadikp}
	(I-T(\alpha,\beta,\omega))\bm u=P(\alpha,\beta,\omega)^{-1}Q\bm u=\bm c,
\end{equation}
where $\bm c=P(\alpha,\beta,\omega)^{-1}\bm b$ and $P(\alpha,\beta,\omega)$ is a preconditioner. This equivalent system can be solved by GMRES \cite{saad1986gmres}. 
Algorithm \ref{alg:gmres_gadikp} gives the solving process for \eqref{eq:linear_system} by putting $P(\alpha,\beta,\omega)$ as a preconditioner in GMRES. 
Note that using the MSKP preconditioner within GMRES requires solving a linear system $P(\alpha,\beta,\omega)\bm v=\bm r$ at each iteration, 
which can be economically solved by Algorithm \ref{alg:gadigkp}. 

\begin{algorithm}
	\caption{GMRES-MSKP method for solving linear system \eqref{eq:linear_system}.}
	\label{alg:gmres_gadikp}
	\KwData{$\bm u^{(0)}\in \mathbb{R}^{n^dm}$, $\alpha,~\beta>0,~0\leq\omega<2$}
	Calculate preconditioner $P(\alpha,\beta,\omega)$ according to \eqref{eq:gadikp_preconditioner}\\
	\For{$k=1,2,\cdots$ until $\{\bm u_k\}$ converges}{
	Calculate $\bm r_0=P(\alpha,\beta,\omega)^{-1}(\bm b-Q\bm u_0)$ and $\bm v_1=\bm r_0/\|\bm r_0\|_2$\\
	\For{$j=1,2,\cdots,k$}{
	$\widetilde{\bm w}_j=Q\bm v_j$\\
	$\bm w_j=P(\alpha,\omega)^{-1}\widetilde{\bm w}_j$\\
	\For{$i=1,2,\cdots,j$}{
	$h_{ij}=(\bm w_j,\bm v_i)$
	}
	$\widehat{\bm v}_{j+1}=\bm w_j-\sum_{i=1}^jh_{i,j}\bm v_i$\\
	$h_{j+1,j}=\|\widehat{\bm v}_{j+1}\|_2$\\
	$\bm v_{j+1}=\widehat{\bm v}_{j+1}/h_{j+1,j}$
	}
	$\bm u_k=\bm u_0+V_k\bm y_k$, where $\bm y_k$ is the solution of $\min_y \|\|\bm r_0\|_2\bm e_1-H_k\bm y\|_2$
	}
\end{algorithm}

Assume that all the eigenvalues of $B^{-1}A$ and $M^{-1}K$ have positive and nonnegative real parts, respectively. $\alpha>0$, $0<\beta\leq \tau \min_{\xi\in \sigma(M^{-1}K)}Re(\xi)$ and $0\leq\omega<2$. Since $\rho(P(\alpha,\beta,\omega)^{-1}Q)=\rho(I-T(\alpha,\beta,\omega))\leq1$, we can conclude that all the eigenvalues of the preconditioned matrix $P(\alpha,\beta,\omega)^{-1}Q$ are located in a circle with radius $1$. Compared with KPS and GKPS preconditioners (see Figure \ref{fig:eig_dist}), MSKP preconditioner can further reduce the condition number of coefficient matrix $Q$, and the eigenvalue distribution of preconditioned system has a more tighter bound.
\begin{figure}[!hbpt]
\centering
	\subfigure[No preconditioner] {\includegraphics[width=3.6cm]{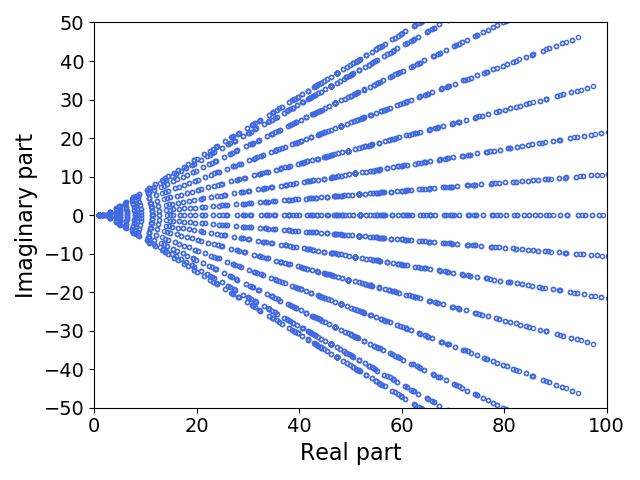}}
	\subfigure[KPS preconditioner] {\includegraphics[width=3.6cm]{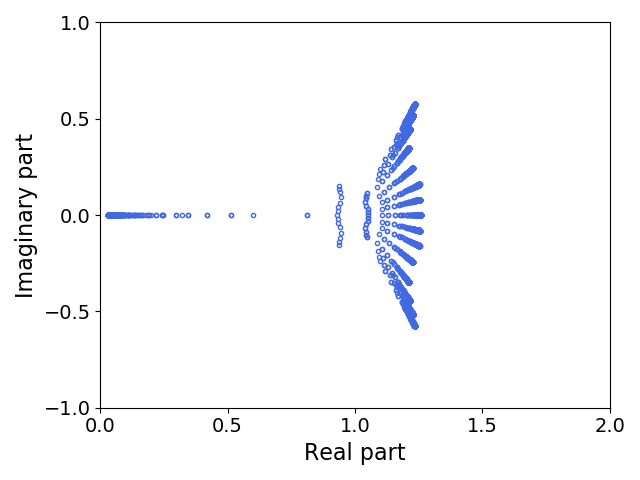}}
	\subfigure[GKPS preconditioner] {\includegraphics[width=3.6cm]{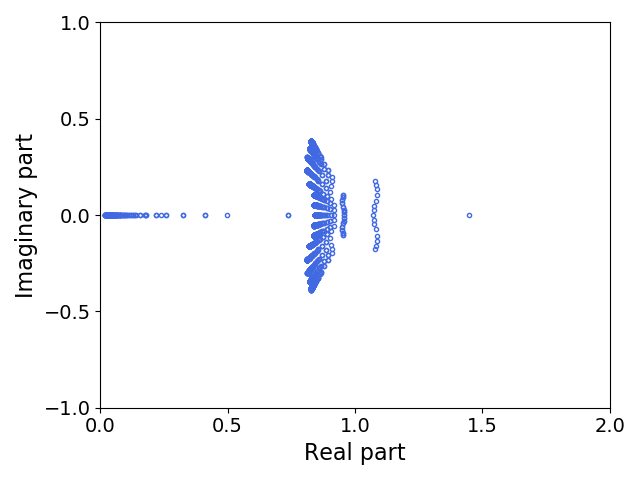}}
	\subfigure[MSKP preconditioner] {\includegraphics[width=3.6cm]{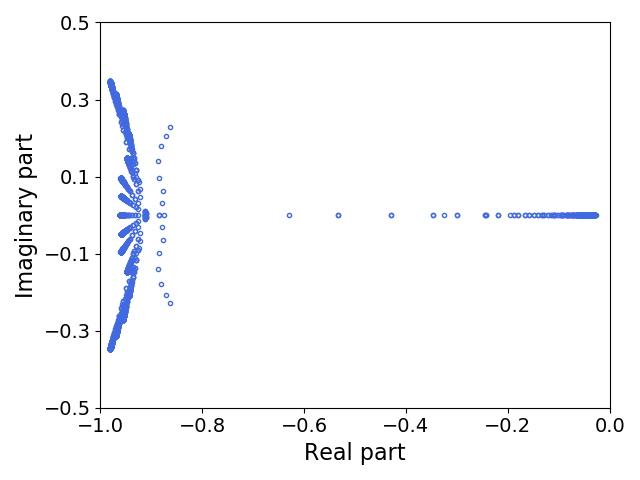}}
	\caption{The eigenvalue distribution of $Q$ with different preconditioners. $Q$ is the coefficient matrix generated by discretizing 2D diffusion equation with $h=\tau=1/16$ (see subsection \ref{subsec:2D_diff} for details). }
\label{fig:eig_dist}
\end{figure}

\section{Numerical experiments}
\label{sec:rslts}

In this section, we present three numerical examples, a 2D diffusion equation, a 2D convection-diffusion equation, and a differential Sylvester matrix equation, to show the performance of MSKP and GMRES with MSKP preconditioner (GMRES-MSKP) methods. As a comparison, we also give experimental results of KPS, GKPS, GMRES, and GMRES-GKPS methods. 
All computations are carried out using PyCharm 2020.2 on a Mac laptop with 2.3 GHz Quad Intel Core i5. 
All tests are started with zero vector. 
All iteration methods are terminated if the relative residual error satisfies $\mbox{RES}=\|\bm r^{(k)}\|_2/\|\bm r^{(0)}\|_2\leq 10^{-6}$, where $\bm r^{(k)}=\bm b-Q\bm u^{(k)}$ is the $k$-step residual. ``IT" and ``CPU" denote the required iterations and the CPU time (in seconds), respectively.
In the following calculations, the relatively optimal parameters of KPS, GKPS, and GMRES-GKPS methods are obtained by traversing approach, and the ones of MSKP and GMRES-MSKP methods are obtained by MTKL method.

We use the MTKL method with the kernel library $\mathcal{K}$ contains
linear, Gaussian, periodic kernels, and their multiplicative combinations.
Concretely, three basic kernel functions are
\begin{align*}
\mbox{Linear kernel:} ~~~& k_l(x,x')=\sigma^2_f(x-c)(x'-c),
\\
\mbox{Gaussian kernel:} ~~~& k_g(x,x')=\sigma^2_f\exp\left(-\frac{\|x-x'\|^2}{2\iota^2}\right),
\\
\mbox{Periodic kernel:} ~~~& 			k_p(x,x')=\sigma^2_f\exp\left(-\frac{2}{\iota^2}\sin^2\Big(\pi \frac{x-x'}{p}\Big) \right).
\end{align*}
The output variance $\sigma^2_f$ determines the average distance of the kernel function away from its mean.
The offset $c$ determines the $x$-coordinate of a point, at which the kernel function has zero variance. 
The lengthscale $\iota$ determines the length of the ‘wiggles'.
The period $p$ measures the distance between repititions of the kernel function. The multiplicative combination kernels are
$k_{ll} = k_l k_l$, $k_{lg}=k_l k_g$, $k_{lp}= k_l k_g$, and $k_{gp}= k_g k_p$.
Consequently, the kernel library is $\mathcal{K}=\{k_l, k_g, k_p, k_{ll}, k_{lg}, k_{lp}, k_{gp} \}$.

\subsection{2D diffusion equation}
\label{subsec:2D_diff}

We consider a 2D diffusion equation
\begin{equation}\label{eq:2Ddiffusion}
	u_t=u_{xx}+u_{yy}+f, \quad (x,y)\in [0,1]^2, ~~ t\in [0,1],
\end{equation}
with homogeneous Dirichlet boundary condition. The exact solution is
\begin{equation*}
	u(x,y,t)=\sin(5.25\pi t)xy(1-x)(1-y)
\end{equation*}
and $f$ is correspondingly determined. 
First, using centered difference scheme on an $n\times n$ uniform grid with mesh size $h=1/(n+1)$ in each unit square, we obtain an $n^2$-dimensional ODE \eqref{eq:ode}, in which $K=(I_n\otimes T_n+T_n\otimes I_n)/h^2$ and $M$ is an $n^2\times n^2$ identity matrix. $T_n=\text{tridiag}(-1,2,-1)\in \mathbb{R}^{n\times n}$. 
Further, we discrete this ODE by the fifth-order generalized Adams method (GAM-5) in BVM framework~\cite{brugnano1998solving} using uniform time grid on $[0, 1]$ with time step size $\tau=1/(m-1)$. 
Then we obtain the full discretize linear system as \eqref{eq:linear_system}, in which the coefficient matrices are
\begin{equation*}
\begin{aligned}
	&A=\begin{bmatrix}
		1\\
		-1 & 1 \\
		& -1 & 1\\
		& & \ddots & \ddots \\
		& & & -1 & 1\\
		& & & & -1 & 1\\
		& & & & & -1 & 1\\
	\end{bmatrix},\\
	&B=\frac{1}{720}
	\begin{bmatrix}
		0\\
		251 & 646 & -264 & 106 & -19\\
		-19 & 346 & 456 & -74 & 11\\
		& \ddots & \ddots & \ddots & \ddots & \ddots\\
		& & -19 & 346 & 456 & -74 & 11\\
		& & 11 & -74 & 456 & 346 & -19\\
		& & -19 & 106 & -264 & 646 & 251\\
	\end{bmatrix}.
\end{aligned}
\end{equation*}

\subsubsection{Predicting relatively optimal parameters}

We use MTKL approach to predict the splitting parameters $(\alpha,\beta,\omega)$ of MSKP method. Table \ref{tab:diff_training_set} gives the training, test, and retrained data sets of $(\alpha,\beta,\omega)$. 
Concretely, $(\alpha,\beta,\omega)$ in the training data set is produced by traversing approach as common matrix splitting methods done but for small scale linear systems, $m$ from 10 to 128 with different step size $\Delta m$. $\alpha$ and $\beta$ are obtained by traversing interval $(0,5]$ with a step size of 0.01. $\omega$ is obtained by traversing interval $[0,2)$ with the same step size. 
For the test data set, let $m$ from 1 to 500 with $\Delta m=1$. 
To predict the parameter more accurately and improve the generation ability, we put the predicted data into the training set to form the retrained data set. In this experiment for the retrained data set, we let $m$ from 128 to 500 with $\Delta m=30$.

\begin{table}[!hbpt]
\centering
\caption{Training, test, and retrained data sets $(\alpha, \beta, \omega)$ in MTKL approach for solving the 2D diffusion equation \eqref{eq:2Ddiffusion} with MSKP method ($n=16$).}
\begin{tabular}{|c|c|c|}
	\hline
	Training set & Test set & Retrained set\\
	\hline
	 $m: 10\thicksim 32,\quad \Delta m=2$& \multirow{3}*{$m: 1\thicksim 500,\quad  \Delta m=1$} & \multirow{3}*{$m: 128\thicksim 500,\quad  \Delta m=30$}\\
	 $m: 36\thicksim 80,~~ \Delta m=4$ & & \\
	 $m: 88\thicksim 128,~~ \Delta m=8$ & & \\
	\hline
\end{tabular}
\label{tab:diff_training_set}
\end{table}

\begin{figure}[hbtp]
	\centering
	\includegraphics[width=4.8cm]{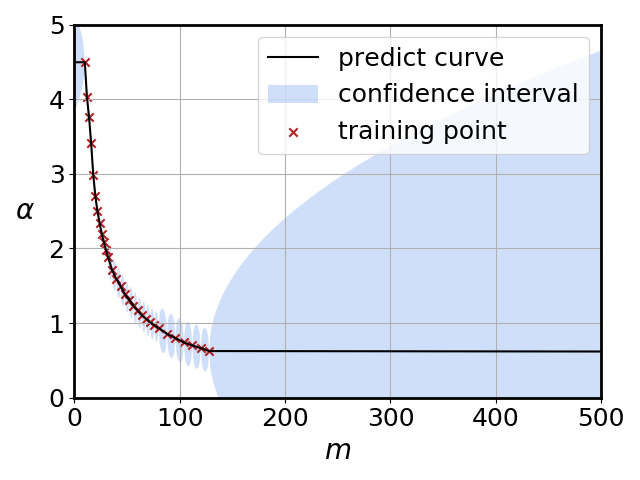}
	\includegraphics[width=4.8cm]{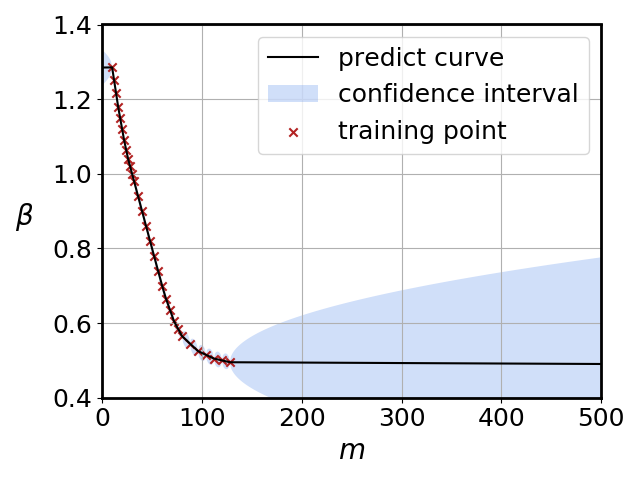}
	\includegraphics[width=4.8cm]{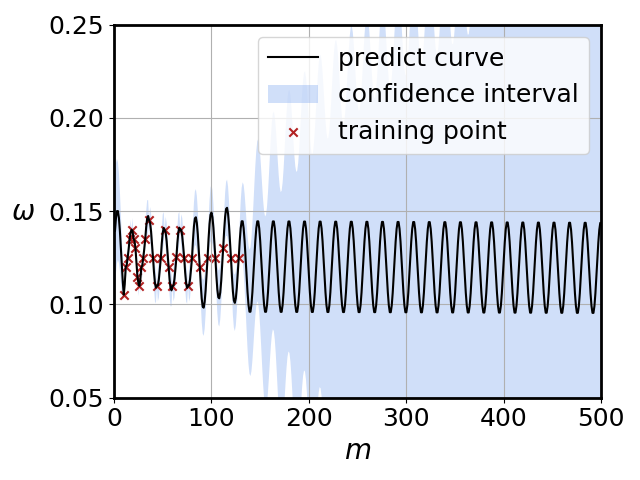}
	\includegraphics[width=4.8cm]{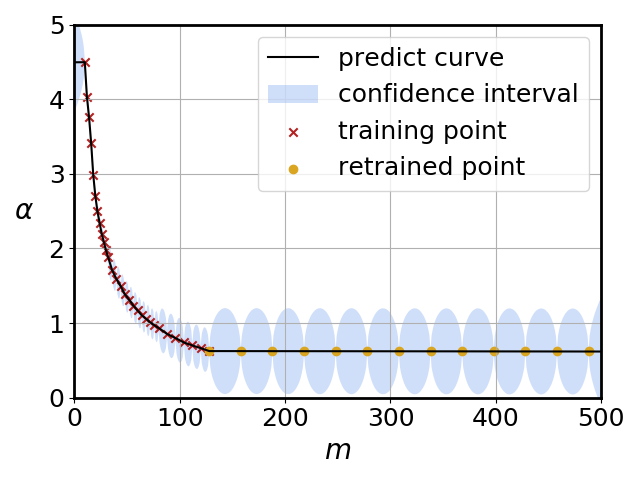}
	\includegraphics[width=4.8cm]{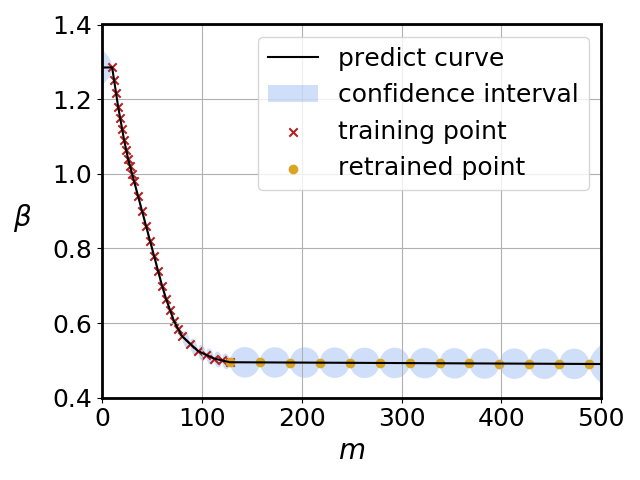}
	\includegraphics[width=4.8cm]{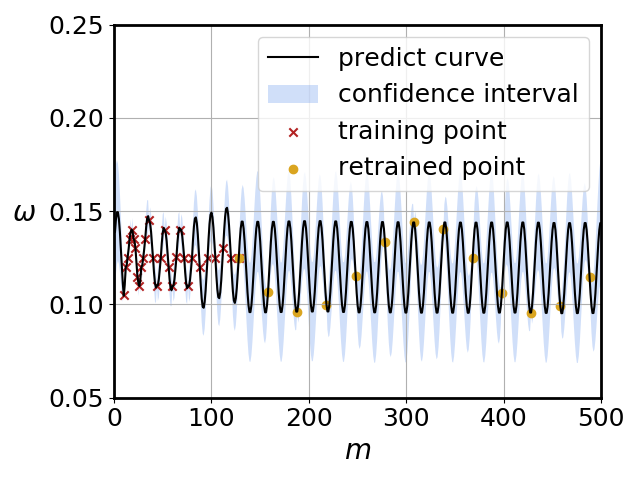}
	\caption{Regression curves of $(\alpha,\beta,\omega)$ against $m$ (fixed $n=16$) when solving the 2D diffusion equation \eqref{eq:2Ddiffusion} with MSKP method. Up: no retraining; Bottom: retraining.}
	\label{fig:diff_gpr_curve}
\end{figure}

In this case, the MTKL method considers a linear combination of
$k_{g}$, $k_p$, $k_{gp}$ in kernel library $\mathcal{K}$. Concretely, for the $l$-th task,
\begin{equation*}
	k_{l}(x,x')=c_{l1}k_g(x,x')+c_{l2}k_p(x,x')+c_{l3}k_{gp}(x,x'),\quad l=1,2,3.
\end{equation*}
Figure \ref{fig:diff_gpr_curve} shows the relative optimal parameter regression curve of $(\alpha, \beta, \omega)$ against $m$, 
and the retrained data set can shrink the confidence interval.
The optimized hyperparameters are
\begin{equation*}
	C=\begin{bmatrix}
		1.07 & 0.02 & 0.13\\
		0.93 & 0.01 & 0.06\\
		0.31 & 0.66 & 0.47\\
	\end{bmatrix},\quad
	K^t=\begin{bmatrix}
		15.91 & 0.00 & 0.00\\
		0.00 & 1.07 & 0.01\\
		0.00 & 0.01 & 1.22\\
	\end{bmatrix},
\end{equation*}
and $\iota = 139.60,~p=8.03$.
Note that, three kernels $k_g(x,x')$, $k_p(x,x')$, and $k_{gp}(x,x')$ in the combination all play important roles. For comparison, Figure \ref{fig:diff_gpr_curve_bad} demonstrates that when only periodic kernel is considered, the GRP method~\cite{jiang2022general} fails to predict accurate splitting parameter $\omega$.

\begin{figure}[hbtp]
	\centering
	\includegraphics[width=0.4\textwidth]{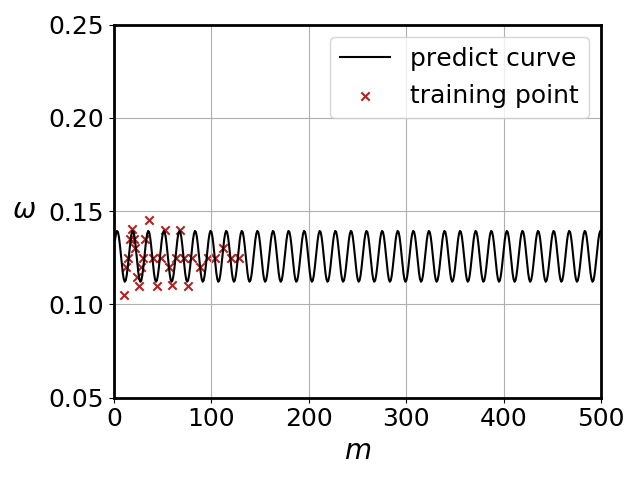}
	\caption{Regression curve of $\omega$ against $m$ (fixed $n=16$) when solving the 2D diffusion equation \eqref{eq:2Ddiffusion} with MSKP method. Here we use GRP method with only periodic kernel.}
	\label{fig:diff_gpr_curve_bad}
\end{figure}

\subsubsection{Comparison}

In this subsection, we compare our methods with KPS, GKPS, GMRES, and GMRES-GKPS methods for solving the 2D diffusion equation \eqref{eq:2Ddiffusion}. 
Table \ref{tab:diff_compare} compares the numerical results with different discrete resolution. 
MSKP method can achieve a convergence efficiency two to three times faster than KPS method, and is superior to GKPS method in terms of iteration steps and computational cost. 
Additionally, MSKP method as a preconditioner can better accelerate GMRES methods compared with GKPS method. As the dimension of system increases, the advantages of our approaches becomes significantly.
Table \ref{tab:diff_traversalcpu} shows the traversal CPU time of different methods when $h=\tau=1/64$. 
Due to the MTKL approach, MSKP and GMRES-MSKP methods do not consume traversal time in selecting relatively optimal parameters. 
However, the traversal time of KPS, GKPS, and GMRES-GKPS methods increases dramatically as the dimension of system becomes larger, resulting in an unaffordable computational cost.

\begin{table}[h]
\centering
\caption{Results of solving the 2D diffusion equation \eqref{eq:2Ddiffusion}.}
\resizebox{0.8\textwidth}{28mm}{
\begin{tabular}{|c|c|ccc|ccc|ccc|}
\hline
	\multirow{2}*{Method} & $\tau$ & \multicolumn{3}{c|}{1/16} & \multicolumn{3}{c|}{1/32} & \multicolumn{3}{c|}{1/64}\\
	\cline{2-11}
	& $h$ & 1/16 & 1/32 & 1/64 & 1/16 & 1/32 & 1/64 & 1/16 & 1/32 & 1/64\\
	\hline
	KPS
		& IT & 33 & 36 & 37 & 43 & 51 & 54 & 45 & 74 & 86\\
		& CPU & 0.67 & 1.58 & 5.56 & 1.28 & 3.93 & 14.56 & 1.94 & 8.20 & 37.79\\
	\hline
	GKPS
		& IT & 19 & 19 & 19 & 22 & 22 & 22 & 36 & 36 & 36\\
		& CPU & 0.38 & 0.83 & 2.89 & 0.64 & 1.71 & 5.94 & 1.55 & 4.02 & 15.98\\
	\hline
	\pmb{MSKP} 		
	    & \pmb{IT} & \pmb{15} & \pmb{15} & \pmb{15} & \pmb{18} & \pmb{19} & \pmb{20} & \pmb{23} & \pmb{26} & \pmb{27}\\ 
	& \pmb{CPU} & \pmb{0.29} & \pmb{0.66} & \pmb{2.25} & \pmb{0.52} & \pmb{1.47} & \pmb{5.40} & \pmb{0.99} & \pmb{2.86} & \pmb{11.91}\\
	\hline \hline
	GMRES		
		& IT & 138 & 336 & 813 & 167 & 401 & 966 & 214 & 511 & 1227\\
		& CPU & 0.32 & 2.55 & 40.09 & 0.51 & 11.81 & 152.48 & 1.18 & 30.31 & 611.77\\
	\hline
    GMRES-GKPS		
		& IT & 12 & 12 & 12 & 19 & 19 & 19 & 33 & 33 & 33\\
		& CPU & 0.18 & 0.42 & 2.31 & 0.57 & 1.37 & 7.26 & 2.15 & 4.29 & 25.03\\
	\hline 
	\pmb{GMRES-MSKP}		
		& \pmb{IT} & \pmb{11} & \pmb{11} & \pmb{12} & \pmb{16} & \pmb{17} & \pmb{18} & \pmb{21} & \pmb{24} & \pmb{25}\\ 
	& \pmb{CPU} & \pmb{0.16} & \pmb{0.39} & \pmb{2.28} & \pmb{0.48} & \pmb{1.22} & \pmb{6.83} & \pmb{1.37} & \pmb{3.12} & \pmb{18.94}\\
	\hline
\end{tabular}
}
\label{tab:diff_compare}
\end{table}

\begin{table}[h]
\centering
\caption{Traversal CPU in solving the 2D diffusion equation \eqref{eq:2Ddiffusion} when $h=\tau=1/64$.}
\resizebox{0.8\textwidth}{5mm}{
\begin{tabular}
{|c|c|c|c|c|c|c|}
	\hline
	Method & KPS & GKPS & \pmb{MSKP} & GMRES & GMRES-GKPS & \pmb{GMRES-MSKP}\\
	\hline
	Traversal CPU & 18025.83 & 5848.68 & \pmb{0} & 0 & 9160.98 & \pmb{0}\\
	\hline	
\end{tabular}
}
\label{tab:diff_traversalcpu}
\end{table}

\subsection{2D convection-diffusion equation}
\label{subsec:2DconvDiff}

The second example is a 2D convection diffusion equation
\begin{equation}\label{eq:2DconvDiff}
	u_t+u_x=u_{xx}+u_{yy}+f,\quad (x,y)\in [0,1]^2, ~~ t\in [0,1],
\end{equation}
with homogeneous Dirichlet boundary condition. $\bm{u}\in \mathbb{R}^{mn^2}$ is the unknown vector of discretizing $u(x,y,t)$. 
$\bm{f}\in \mathbb{R}^{mn^2}$ is the vector of discretizing $f(x,y,t)$, which determined by choosing the exact solution $\bm{u}_e = (1,1,\cdots,1)^{T}$. 
Using centered difference scheme on an $n\times n$ uniform grid with mesh size $h=1/(n+1)$ in each unit square, 
we obtain an $n^2$-dimensional ODE \eqref{eq:ode}, in which $K=I_n\otimes P_n+Q_n\otimes I_n$ and $M$ is an $n^2\times n^2$ identity matrix. 
$P_n=\text{tridiag}(-1, 2, -1)/h^2$ and $Q_n=\text{tridiag}(-1/2h-1/h^2,2/h^2,1/2h-1/h^2) \in \mathbb{R}^{n\times n}$. Then we apply GAM-5 on uniform time grid with a time step size $\tau=1/(m-1)$ to discrete this ODE. Finally, we obtain a linear system of the form as \eqref{eq:linear_system}.

\subsubsection{Predicting relatively optimal parameters}

We use MTKL approach to predict the splitting parameter $(\alpha,\beta,\omega)$ of MSKP method. Table \ref{tab:diff_conv_training_set} gives the training, test, and retrained data sets of $(\alpha,\beta,\omega)$. Again, $(\alpha,\beta,\omega)$ in the training data set is generated by traversing method, $m$ from 10 to 128 with different step size $\Delta m$. 
The traversing interval and step size of $\alpha$, $\beta$, and $\omega$ are the same as the first example. 
For the test data set, let $m$ from 1 to 500 with $\Delta m=1$. And for the retrained data set, let $m$ from 128 to 500 with $\Delta m=30$.
	
\begin{table}[h]
\centering
\caption{Training, test, and retrained data sets $(\alpha, \beta, \omega)$ in MTKL approach for solving the 2D convection-diffusion equation \eqref{eq:2DconvDiff} with MSKP method ($n=16$).}
\resizebox{0.8\textwidth}{9mm}{
\begin{tabular}{|c|c|c|}
\hline
	Training set & Test set & Retrained set\\
	\hline
	 $m: 10\thicksim 32,\quad \Delta m=2$& \multirow{3}*{$m: 1\thicksim 500,\quad  \Delta m=1$} & \multirow{3}*{$m: 128\thicksim 500,\quad  \Delta m=30$}\\
	 $m: 36\thicksim 80,\quad \Delta m=4$ & & \\
	 $m: 88\thicksim 128,\quad \Delta m=8$ & & \\
\hline
\end{tabular}
\label{tab:diff_conv_training_set}
}
\end{table}

In this case, the MTKL method uses a linear combination of
$k_{g}$, $k_p$, $k_{gp}$ to form the kernel library $\mathcal{K}$. Concretely, for the $l$-th task,
\begin{equation*}
	k_l(x,x')=c_{l1}k_g(x,x')+c_{l2}k_p(x,x')+c_{l3}k_{gp}(x,x'),\quad l=1,2,3.
\end{equation*}
Figure \ref{fig:diff_conv_gpr_curve} shows the relative optimal parameter regression curve of $(\alpha, \beta, \omega)$ against $m$. The optimized hyperparameters are
\begin{equation*}
	C=\begin{bmatrix}
		1.21 & 0.05 & 0.04\\
		1.17 & 0.03 & 0.02\\
		0.03 & 0.34 & 0.67\\
	\end{bmatrix},\quad
	K^t=\begin{bmatrix}
		0.21 & 0.41 & 8.38\\
		0.41 & 0.33 & 7.65\\
		8.38 & 7.65 & 1.01\\
	\end{bmatrix},
\end{equation*}
and $\iota = 99.54,~p=7.83$.
From Figure \ref{fig:diff_conv_gpr_curve}, it can be seen that using the retrained data set can shrink the confidence interval. This improves the prediction accuracy and strengthens the generalization ability of the regression model.

\begin{figure}[htb]
	\centering
	\includegraphics[width=5cm]{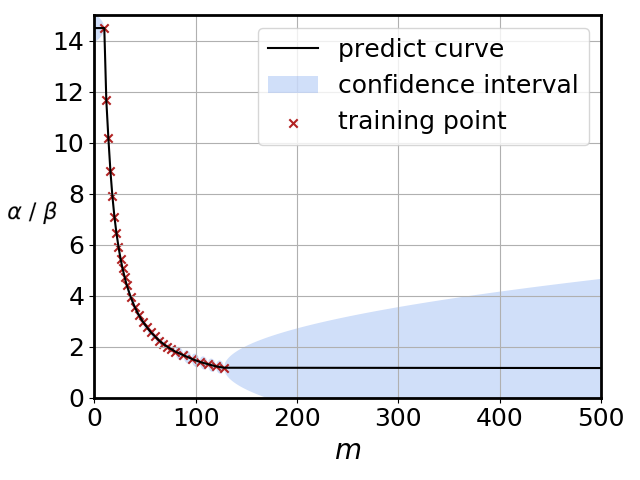}
	\includegraphics[width=5cm]{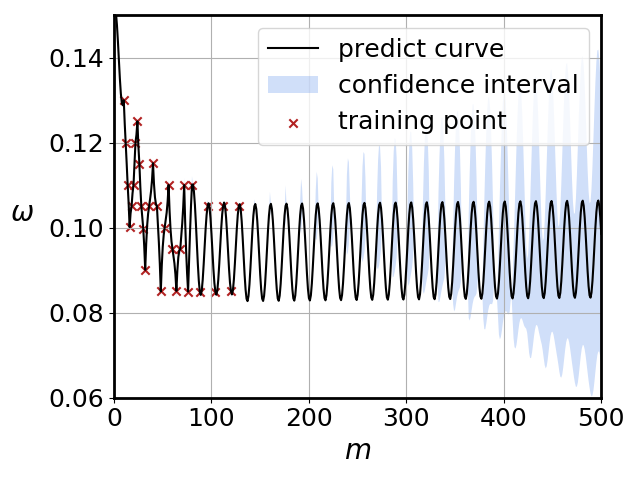}	
	\includegraphics[width=5cm]{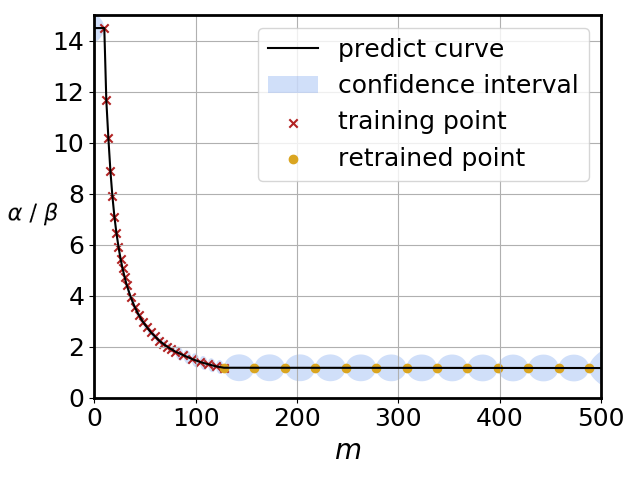}
	\includegraphics[width=5cm]{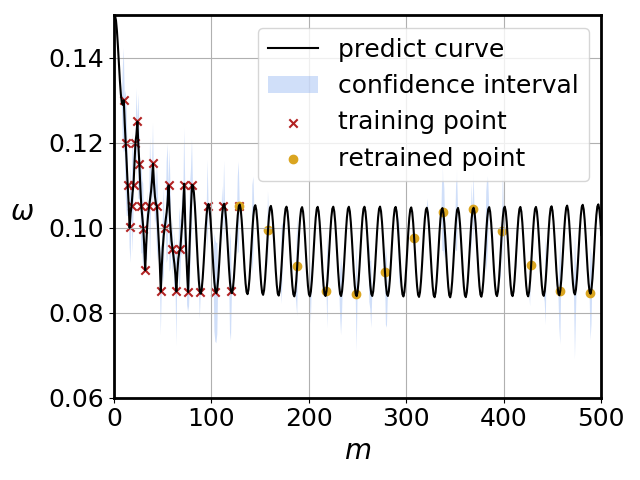}
	\caption{Regression curves of $(\alpha,\beta,\omega)$ against $m$ (fixed $n=16$) when solving 2D convection-diffusion equation \eqref{eq:2DconvDiff} with MSKP method. Up: no retraining; Bottom: retraining.}
	\label{fig:diff_conv_gpr_curve}
\end{figure}

\subsubsection{Comparison}

In this subsection, we compare our methods with KPS, GKPS, GMRES, and GMRES-GKPS methods for solving the 2D convection-diffusion equation \eqref{eq:2DconvDiff}. 
Table \ref{tab:diff_conv_compare} shows MSKP method has a better convergence performance than KPS and GKPS methods. 
In this experiment, $\beta$ in GKPS does not affect the performance, while $\omega$ in MSKP still has a great influence. 
Moreover, MSKP method as a preconditioner can accelerate GMRES method by tens of times.
As Table \ref{tab:diff_conv_traversalcpu} presents, MTKL method demonstrates an immense advantage in predicting the relatively optimal splitting parameters with large-scale systems. 

\begin{table}[h]
\centering
\caption{Results of solving the 2D convection-diffusion equation \eqref{eq:2DconvDiff}.}
\resizebox{0.8\textwidth}{23mm}{
\begin{tabular}{|c|c|ccc|ccc|ccc|}
\hline
	\multirow{2}*{Method} & $\tau$ & \multicolumn{3}{c|}{1/16} & \multicolumn{3}{c|}{1/32} & \multicolumn{3}{c|}{1/64}\\
	\cline{2-11}
	& $h$ & 1/16 & 1/32 & 1/64 & 1/16 & 1/32 & 1/64 & 1/16 & 1/32 & 1/64\\
	\hline
	KPS / GKPS & IT & 58 & 101 & 178 & 59 & 98 & 169 & 56 & 103 & 168\\
		& CPU & 0.87 & 2.56 & 11.21 & 1.54 & 4.95 & 22.69 & 2.69 & 10.01 & 46.74\\
	\hline
	\pmb{MSKP}		
		& \pmb{IT} & \pmb{43} & \pmb{68} & \pmb{108} & \pmb{44} & \pmb{70} & \pmb{111} & \pmb{45} & \pmb{73} & \pmb{115}\\
	& \pmb{CPU} & \pmb{0.66} & \pmb{1.73} & \pmb{6.81} & \pmb{1.15} & \pmb{3.48} & \pmb{14.87} & \pmb{2.16} & \pmb{7.16} & \pmb{31.93}\\
	\hline \hline
	GMRES 		
		& IT & 150 & 333 & 731 & 158 & 350 & 764 & 197 & 433 & 944\\
		& CPU & 0.41 & 3.22 & 33.66 & 0.59 & 6.80 & 82.85 & 1.26 & 11.43 & 216.34\\
	\hline
	GMRES-GKPS 		
		& IT & 25 & 25 & 26 & 30 & 34 & 45 & 46 & 51 & 60\\
		& CPU & 0.35 & 0.67 & 3.04 & 0.84 & 2.18 & 12.44 & 2.69 & 6.58 & 30.90\\
	\hline
	\pmb{GMRES-MSKP}		
		& \pmb{IT} & \pmb{22} & \pmb{24} & \pmb{24} & \pmb{27} & \pmb{31} & \pmb{33} & \pmb{32} & \pmb{37} & \pmb{40}\\
	& \pmb{CPU} & \pmb{0.31} & \pmb{0.64} & \pmb{2.76} & \pmb{0.75} & \pmb{1.98} & \pmb{9.18} & \pmb{1.87} & \pmb{4.77} & \pmb{21.08}\\
	\hline
\end{tabular}
}
\label{tab:diff_conv_compare}
\end{table}

\begin{table}[h]
\centering
\caption{Traversal CPU in solving the 2D convection-diffusion equation \eqref{eq:2DconvDiff} when $h=\tau=1/64$.}
\resizebox{0.8\textwidth}{5mm}{
\begin{tabular}{|c|c|c|c|c|c|}
	\hline
	Method & KPS/GKPS  & \pmb{MSKP} & GMRES & GMRES-GKPS & \pmb{GMRES-MSKP}\\
	\hline
	Traversal CPU & 23276.52  & \pmb{0} & 0 & 12324.91 & \pmb{0}\\
	\hline
\end{tabular}
}
\label{tab:diff_conv_traversalcpu}
\end{table}

\subsection{Differential Sylvester matrix equation}
\label{subsec:matrixeqn}

Finally, we consider a differential Sylvester matrix equation
\eqref{eq:linear_matrix_system},  
in which $\mathcal{A}$ and $\mathcal{B}$ obtained from the centered finite difference discretization of the operator $\mathcal{L}(u)=\Delta u+f_1(x,y)\frac{\partial u}{\partial x}+f_2(x,y)\frac{\partial u}{\partial y}+f(x,y)u$ on the unit square $[0,1]^2$ with homogeneous Dirichlet boundary condition. 
The number of inner grid points in each direction  is $n_0$, then the dimension of $\mathcal{A}$ and $\mathcal{B}$ is $n=n_0^2$. For this experiment, we extract the $\mathcal{A}$ and $\mathcal{B}$ from the Lyapack package \cite{penzl2000matlab} using the command $\text{fdm\_2d\_matrix}(n_0,x,y,0)$. $\mathcal{E},~ \mathcal{F}\in \mathbb{R}^{n\times 2}$ are obtained by the random values uniformly distributed on [0,1]. The initial condition is $X_0=\bm 0$. We discrete the differential Sylvester matrix equation \eqref{eq:linear_matrix_system} in temporal direction by GAM-5 with uniform time grid on $[0, 1]$ (time step size  $\tau=0.1$), and obtain a similar linear system of the form as \eqref{eq:linear_system}.

\subsubsection{Predicting relatively optimal parameters}

We use MTKL method to predict the relatively optimal splitting parameters $(\alpha,\beta,\omega)$ of MSKP scheme. Table \ref{tab:sylvester_training_set} gives the training, test, and retrained data sets of $(\alpha,\beta,\omega)$. $(\alpha,\beta,\omega)$ in the training data set is obtained by a traversing way for small scale linear systems, $n$ from 16 to 225 with different step size. For the test data set, let $n$ from 1 to 1200 with $\Delta n=1$. And for the retrained data set, let $n$ from 250 to 1200 with $\Delta n=50$.

\begin{table}[h]
\centering
\caption{Training, test, and retrained data sets $(\alpha, \beta, \omega)$ in MTKL approach for solving the differential Sylvester matrix equation \eqref{eq:linear_matrix_system} with MSKP method ($\tau=0.1$).}
\resizebox{0.8\textwidth}{8mm}{
\begin{tabular}{|c|c|c|}
\hline
	Training set & Test set & Retrained set\\
	\hline
	 $n=\gamma^2,\gamma\in \Gamma$& \multirow{2}*{$n: 1\thicksim 1200,  \Delta n=1$} & \multirow{2}*{$n: 250\thicksim 1200,  \Delta n=50$}\\
	 $\Gamma =\{\gamma|\gamma\in \mathbb{Z}, 4\leq \gamma \leq 15 \} $ & & \\
\hline
\end{tabular}
}
\label{tab:sylvester_training_set}
\end{table}

In this case, the kernel library $\mathcal{K}$ in MTKL method is a linear combination of $k_{g}$, $k_l$, $k_{gl}$. For the $l$-th task,
\begin{equation*}
	k_{l}(x,x')=c_{l1}k_g(x,x')+c_{l2}k_l(x,x')+c_{l3}k_{gl}(x,x'),\quad l=1,2,3.
\end{equation*}
Figure \ref{fig:sylvester_gpr_curve} shows the relative optimal parameter regression curve of $(\alpha, \beta, \omega)$ against $m$. The optimized hyperparameters are
\begin{equation*}
	C=\begin{bmatrix}
		0.89 & 0.05 & 0.24\\
		0.99 & 0.11 & 0.07\\
		0.32 & 0.26 & 0.79\\
	\end{bmatrix},\quad
	K^t=\begin{bmatrix}
		1.23 & 0.03 & 0.01\\
		0.03 & 1.14 & 0.12\\
		0.01 & 0.12 & 0.97\\
	\end{bmatrix},\quad \iota = 1.15.
\end{equation*}

\begin{figure}[hbtp]
	\centering
	\includegraphics[width=4.8cm]{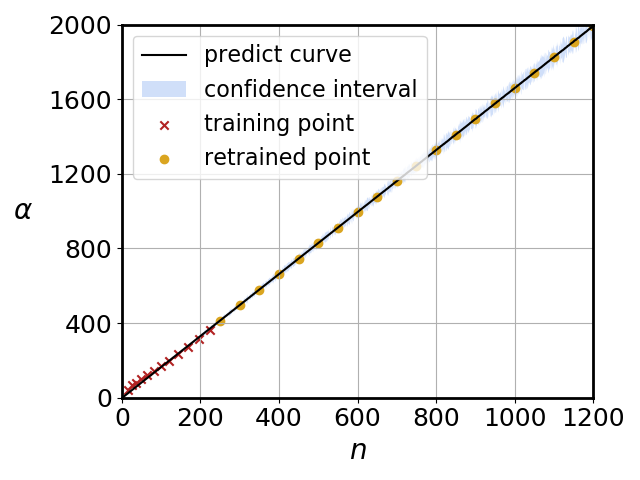}
	\includegraphics[width=4.8cm]{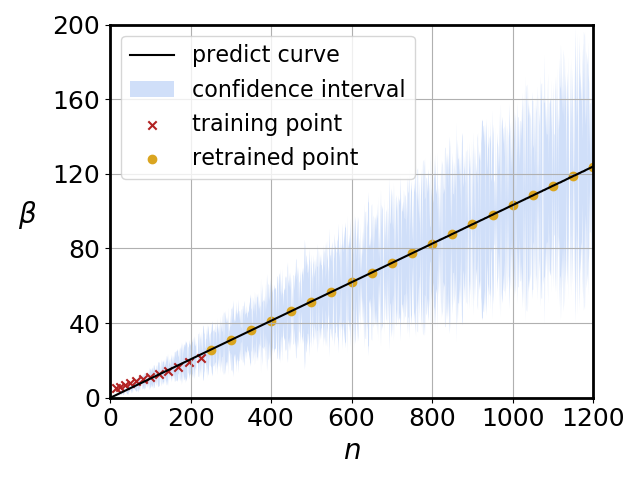}
	\includegraphics[width=4.8cm]{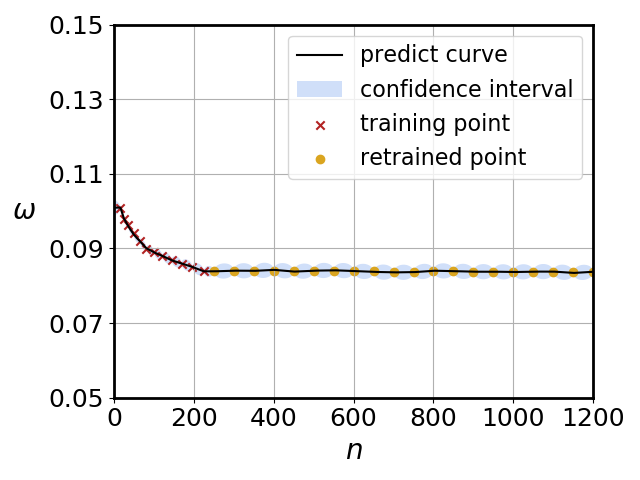}
	\caption{Regression curves of $(\alpha,\beta,\omega)$ against $n$ (fixed $\tau=0.1$) when solving the differential Sylvester matrix equation \eqref{eq:linear_matrix_system} with MSKP method.}
	\label{fig:sylvester_gpr_curve}
\end{figure}

\subsubsection{Comparison}

We compare our proposed methods with KPS, GKPS, and GMRES approaches for solving the differential Sylvester matrix equation \eqref{eq:linear_matrix_system}. Table \ref{tab:sylvester_compare} shows corresponding numerical results, which present that our approach has a significant advantage for solving linear matrix systems. 
When the scale of system is larger than one hundred thousand, MSKP method as a preconditioner can speed up tens to hundreds times over GMRES.

\begin{table}[h]
\centering
\caption{Results of solving the differential Sylvester matrix equation \eqref{eq:linear_matrix_system} with $\tau=0.1$.}
\resizebox{0.8\textwidth}{22mm}{
\begin{tabular}{|c|c|c|c|c|c|c|c|c|}
\hline
	\multirow{2}*{$n^2$} & \multicolumn{2}{c|}{KPS} & \multicolumn{2}{c|}{GKPS} &  \multicolumn{2}{c|}{\pmb{MSKP}} & \multicolumn{2}{c|}{GMRES} \\
	\cline{2-9}
	& IT & CPU & IT & CPU & IT & CPU & IT & CPU \\
	\hline
	256	& 19 & 0.29 & 9 & 0.14 & \pmb{4} & \pmb{0.06} & 70 & 0.11 \\
	\hline
	1296 & 27 & 0.62 & 9 & 0.21 & \pmb{4} & \pmb{0.09} & 114 & 0.56 \\
	\hline
	4096 & 40 & 2.47 & 9 & 0.53 & \pmb{4} & \pmb{0.23} & 160 & 2.35 \\
	\hline
	20736 & 68 & 28.76 & 9 & 3.74 & \pmb{4} & \pmb{1.65} & 256 & 13.56 \\
	\hline
	65536 & 101 & 146.21 & 9 & 12.98 & \pmb{4} & \pmb{5.75} & 362 & 193.77 \\
	\hline
	390625 & 173 & 972.26 & 9 & 106.77 & \pmb{4} & \pmb{45.62} & 491 & 2597.43 \\
	\hline
	1048576 & 259 & 6011.62 & 9 & 619.08 & \pmb{4} & \pmb{247.67} &  & $>30000$ \\
	\hline
\end{tabular}
}
\label{tab:sylvester_compare}
\end{table}

\section{Conclusion and future work}
\label{sec:conclusion}

This paper has developed a new parameter selection method for matrix splitting iteration methods.
Concretely, we present the MTKL method to address the problems of multi-parameter optimization and kernel selection. 
Moreover, we propose a new matrix splitting iteration scheme, MSKP method,  
to solving the TDLSs \eqref{eq:time-dependent_linear_system}. We give its convergence analysis and preconditioning strategy. 
To demonstrate the efficiency of our approaches, we solve a 2D diffusion equation, a 2D convection-diffusion equation, and a differential Sylvester matrix equation. 
Numerical results illustrate MTKL method has a huge advantage in predicting the relatively optimal splitting parameters, especially for large-scale systems. 
Further, the MSKP method as a preconditioner can effectively accelerate GMRES. 
Especially for solving the differential Sylvester matrix equation, the speedup ratio can reach tens to hundreds of times when the scale of the system is larger than one hundred thousand. 

There are still lots of works worthy of our further study.
For instance, the first one is to give the convergence rate analysis of ADI methods. 
The second interesting work is to use machine learning to train iteration schemes and splitting parameters that are consistent with concrete problems. 
The third challenge work is to develop our methods for solving nonlinear systems.

\section*{Acknowledgments}
This work is supported in part by National Natural Science Foundation of China (12171412), Natural Science Foundation for Distinguished Young Scholars of Hunan Province (2021JJ10037), Hunan Youth Science and Technology Innovation Talents Project (2021RC3110), the Key Project of Education Department of Hunan Province (21A0116),  Hunan Provincial Innovation Foundation for Postgraduate (CX20220647).


\end{document}